\documentclass[]{aspm}

\articleinfo{}{}{}

\usepackage{amsmath,amssymb}
\usepackage{amsthm}
\usepackage{bm}
\usepackage{graphicx}
\usepackage{ascmac}

\newtheorem{definition}{Definition}[section]
  \newtheorem{theorem}[definition]{Theorem}
   \newtheorem{lemma}[definition]{Lemma}
  
  \newtheorem{proposition}[definition]{Proposition}
  \newtheorem{claim}[definition]{Claim}

\newtheorem{remark}[definition]{Remark}

\title[On keen Heegaard splittings]
{On keen Heegaard splittings}

\author[A. Ido]{Ayako Ido}
\author[Y. Jang]{Yeonhee Jang}
\author[T. Kobayashi]{Tsuyoshi Kobayashi}

\address[A. Ido]{Department of Mathematics Education, Aichi University of Education\\ 1 Hirosawa, Igaya-cho, Kariya-shi, 448-8542, Japan}
\address[Y. Jang, T. Kobayashi]{Department of Mathematics, Nara Women's University\\ Kitauoya Nishimachi, Nara, 630-8506 Japan}

\email[A. Ido]{ayakoido@auecc.aichi-edu.ac.jp}
\email[Y. Jang]{yeonheejang@cc.nara-wu.ac.jp}
\email[T. Kobayashi]{tsuyoshi@cc.nara-wu.ac.jp}

\rcvdate{Month Day, Year}
\rvsdate{Month Day, Year}

\subjclass[2010]{57M27}

\keywords{Heegaard splitting, curve complex, distance}

\date{\empty}

\begin{document}

\begin{abstract}
In this paper, we introduce a new concept of {\it strongly keen} for Heegaard splittings, and show that, for any integers $n\geq 2$ and $g\geq 3$, there exists a strongly keen Heegaard splitting of genus $g$ whose Hempel distance is $n$. 
\end{abstract}

\maketitle

\section{Introduction}

The {\it curve complex} $\mathcal{C}(S)$ of a compact surface $S$ introduced by Harvey \cite{Ha} has been used to prove many deep results in 3-dimentional topology. 
In particular, Hempel \cite{He} defined the {\it Hempel distance} for a Heegaard splitting $V_{1}\cup_{S} V_{2}$ by $d(S)=d_{S}(\mathcal{D}(V_{1}),\mathcal{D}(V_{2}))=\min\{d_S(x, y)\mid x\in\mathcal{D}(V_{1}), y\in\mathcal{D}(V_{2})\}$,
where $d_{S}$ is the simplicial distance of $\mathcal{C}(S)$ (for the definition, see Section~2), and $\mathcal{D}(V_{i})$ is the disk complex of the handlebody $V_{i}$ ($i=1, 2$). 
There have been many works on Hempel distance. 
For example, some authors showed that the existence of high distance Heegaard splittings (see \cite{AS, E, He}, for example). 
Moreover, it is also shown that there exist Heegaard splittings of Hempel distance exactly $n$ for various integers $n$ (see \cite{BS, IJK, Jo, QZG, Yos}, for example).
Here we note that the pair $(x, y)$  in the above definition that realizes $d(S)$ may not be unique. 
Hence it may be natural to settle: we say that a Heegaard splitting $V_{1}\cup_{S}V_{2}$ is {\it keen} if its Hempel distance is realized by a unique pair of elements of $\mathcal{D}(V_{1})$ and $\mathcal{D}(V_{2})$.
Namely, $V_{1}\cup_S V_{2}$ is keen if it satisfies the following.
\begin{itemize}
\item If $d_S(a,b)=d_S(a',b')=d_S(\mathcal{D}(V_{1}),\mathcal{D}(V_{2}))$ for $a,a'\in\mathcal{D}(V_{1})$ and $b,b'\in\mathcal{D}(V_{2})$, then $a=a'$ and $b=b'$.
\end{itemize}
In Proposition \ref{prop-keen}, we give necessary conditions for a Heegaard splitting to be keen. We note that these show that Heegaard splittings given in \cite{IJK, Jo, QZG} are not keen (Remark \ref{rmk-keen}). 
We also note that Proposition \ref{prop-keen} shows that every genus-2 Heegaard splitting is not keen. 

By the way, for a keen Heegaard splitting $V_{1}\cup_{S}V_{2}$, the geodesics joining the unique pair of elements of $\mathcal{D}(V_{1})$ and $\mathcal{D}(V_{2})$  may not be unique. 
In fact, Johnson \cite{Jo} gives an example of a Heegaard splitting $V_{1}\cup_{S}V_{2}$ such that there is a pair of elements $(x, y)$ of $\mathcal{D}(V_{1})$ and $\mathcal{D}(V_{2})$ realizing the Hempel distance such that there are infinitely many geodesics joining $x$ and $y$.    
We say that a Heegaard splitting $V_{1}\cup_{S} V_{2}$ is {\it strongly keen} if the geodesics joining the pair of elements of $D(V_{1})$ and $D(V_{2})$ are unique. 
The main result of this paper gives the existence of strongly keen Heegaard splitting with Hempel distance $n$ for each $g\geq3$ and $n\geq2$ as follows.

\begin{theorem}\label{thm-1}
For any integers $n\geq 2$ and $g\geq 3$, there exists a strongly keen genus-$g$ Heegaard splitting of Hempel distance $n$.
\end{theorem}


\section{Preliminaries}\label{sec-pre}

Let $S$ be a compact connected orientable surface. 
A simple closed curve in $S$ is {\it essential} if it does not bound a disk in $S$ and is not parallel to a component of $\partial{S}$. 
An arc properly embedded in $S$ is {\it essential} if it does not co-bound a disk in $S$ together with an arc on $\partial{S}$. 

\medskip
\noindent
{\bf Heegaard splittings}

A connected 3-manifold $C$ is a {\it compression-body} if there exists a closed (possibly empty) surface $F$ and a 0-handle $B$ such that $C$ is obtained from $(F\times[0, 1])\cup B$ 
by adding 1-handles to $F \times \{1\}\cup \partial B$. 
The subsurface of $\partial C$ corresponding to $F\times\{0\}$ is denoted by $\partial_{-} C$, and $\partial_{+} C$ denotes the subsurface $\partial C\setminus\partial_{-} C$ of $\partial C$. 
A compression-body $C$ is called a {\it handlebody} if $\partial_- C=\emptyset$.

Let $M$ be a closed orientable 3-manifold. 
We say that $V_{1}\cup_S V_{2}$ is a {\it Heegaard splitting} of $M$ if $V_{1}$ and $V_{2}$ are handlebodies in $M$ such that $V_{1}\cup V_{2}=M$ and $V_1\cap V_2=\partial V_{1}=\partial V_{2}=S$. 
The genus of $S$ is called the {\it genus} of the Heegaard splitting $V_{1}\cup_S V_{2}$. 
Alternatively, given a Heegaard splitting $V_{1}\cup_S V_{2}$ of $M$, we may regard that there is a homeomorphism $f:\partial V_2\rightarrow\partial V_1$ such that $M$ is obtained from $V_1$ and $V_2$ by identifying $\partial V_1$ and $\partial V_2$ via $f$.
When we take this viewpoint, we will denote the Heegaard splitting by the expression $V_{1}\cup_{f}V_{2}$.

\medskip
\noindent
{\bf Curve complexes}

Let $S$ be a compact connected orientable surface with genus $g$ and $p$ boundary components.  
We say that $S$ is {\it sporadic} if $g=0, p\leq 4$ or $g=1, p\leq1$. 
We say that $S$ is {\it simple} if $S$ contains no essential simple closed curves. 
We note that $S$ is simple if and only if $S$ is a 2-sphere with at most three boundary components. 
The {\it curve complex} $\mathcal{C}(S)$ is defined as follows: if $S$ is non-sporadic, each vertex of $\mathcal{C}(S)$ is the isotopy class of an essential simple closed curve on $S$, and a collection of $k+1$ vertices forms a $k$-simplex of $\mathcal{C}(S)$ if they can be realized by mutually disjoint curves in $S$. 
In sporadic cases, we need to modify the definition of the curve complex slightly, as follows. 
Note that the surface $S$ is simple unless $S$ is a torus, a torus with one boundary component, or a sphere with 4 boundary components.
When $S$ is a torus or a torus with one boundary component (resp. a sphere with 4 boundary components), a collection of $k+1$ vertices forms a $k$-simplex of $\mathcal{C}(S)$ if they can be realized by essential simple closed curves in $S$ which mutually intersect exactly once (resp. twice). 
The {\it arc-and-curve complex} $\mathcal{AC}(S)$ is defined similarly, as follows: each vertex of $\mathcal{AC}(S)$ is the isotopy class of an essential properly embedded arc or an essential simple closed curve on $S$, and a collection of $k+1$ vertices forms a $k$-simplex of $\mathcal{AC}(S)$ if they can be realized by mutually disjoint arcs or simple closed curves in $S$. 
The symbol $\mathcal{C}^0(S)$ (resp. $\mathcal{AC}^0(S)$) denotes the 0-skeleton of $\mathcal{C}(S)$ (resp. $\mathcal{AC}(S)$).
Throughout this paper, for a vertex $x\in\mathcal{C}^0(S)$ we often abuse notation and use $x$ to represent (the isotopy class of) a geometric representative of $x$. 

For two vertices $a, b$ of $\mathcal{C}(S)$, we define the {\it distance} $d_{\mathcal{C}(S)}(a, b)$ between $a$ and $b$, which will be denoted by $d_{S}(a, b)$ in brief, as the minimal number of 1-simplexes of a simplicial path in $\mathcal{C}(S)$ joining $a$ and $b$. 
For two subsets $A, B$ of $\mathcal{C}^0(S)$, 
we define ${\rm diam}_{S}(A, B):=$ the diameter of $A\cup B$. 
Similarly, we can define $d_{\mathcal{AC}(S)}(a, b)$ for $a,b\in\mathcal{AC}^0(S)$ and ${\rm diam}_{\mathcal{AC}(S)}(A, B)$ for $A,B\subset\mathcal{AC}^0(S)$.

For a sequence $a_0,a_1,\dots,a_n$ of vertices in $\mathcal{C}(S)$ with $a_i\cap a_{i+1}
=\emptyset$ $(i=0,1,\dots,n-1)$, we denote by $[a_0,a_1,\dots,a_n]$ the path in $\mathcal{C}(S)$ with vertices $a_0,a_1,\dots,a_n$ in this order.
We say that a path $[a_0,a_1,\dots,a_n]$ is a {\it geodesic} if $n=d_S(a_0,a_n)$.

Let $C$ be a compression-body. 
A disk $D$ properly embedded in $C$ is {\it essential} if $\partial D$ is an essential simple closed curve in $\partial_{+}C$.
Then the {\it disk complex} $\mathcal{D}(C)$ is the subset of $\mathcal{C}^0(\partial_{+}C)$ consisting of the vertices with representatives bounding essential disks of $C$. 

For a genus-$g(\geq 2)$ Heegaard splitting $V_{1}\cup_S V_{2}$, the {\it Hempel distance} of $V_{1}\cup_S V_{2}$ is defined by $d_S(\mathcal{D}(V_{1}),\mathcal{D}(V_{2}))=\min\{d_S(x, y)\mid x\in\mathcal{D}(V_{1}), y\in\mathcal{D}(V_{2})\}$.

\medskip
\noindent
{\bf Subsurface projection maps}

Let $\mathcal{P}(Y)$ denote the power set of a set $Y$. 
Let $S$ be a compact connected orientable surface, and let $X$ be a subsurface of $S$.
We suppose that both $S$ and $X$ are non-sporadic, and each component of $\partial X$ is either contained in $\partial S$ or essential in $S$. 
We call the composition $\pi_0\circ\pi_A$ of maps $\pi_A:\mathcal{C}^{0}(S)\rightarrow \mathcal{P}(\mathcal{AC}^{0}(X))$ and $\pi_0:\mathcal{P}(\mathcal{AC}^{0}(X))\rightarrow\mathcal{P}(\mathcal{C}^{0}(X))$ a {\it subsurface projection} if they satisfy the following: for $\alpha \in \mathcal{C}^{0}(S)$, take a representative of $\alpha$ 
so that $|\alpha\cap X|$ is minimal, where $|\cdot|$ is the number of connected components. Then 

\begin{itemize}
\item $\pi_{A}(\alpha)$ is the set of all isotopy classes of the components of $\alpha\cap X$,
\item $\pi_0(\{\alpha_1,\dots,\alpha_n\})$ is the union, for all $i=1,\dots,n$, of the set of all isotopy classes of the components of $\partial N(\alpha_{i}\cup\partial X)$ which are essential in $X$, where $N(\alpha_{i}\cup\partial X)$ is a regular neighborhood of $\alpha_i\cup\partial X$ in $X$.
\end{itemize}
We denote the subsurface projection $\pi_0\circ\pi_A$ by $\pi_X$.
We say that $\alpha$ {\it misses} $X$ (resp. $\alpha$ {\it cuts} $X$) if $\alpha\cap X=\emptyset$ (resp. $\alpha\cap X\neq\emptyset$).

\begin{lemma}(\cite[Lemma 2.2]{MM2})\label{lem-ac-to-c}
Let $X$ be as above. 
Let $A$ and $B$ be subsets of $\mathcal{AC}^0(X)$.
If ${\rm diam}_{\mathcal{AC}(X)}(A, B)\leq 1$, then ${\rm diam}_{X}(\pi_{0}(A), \pi_{0}(B))\leq 2$.
\end{lemma}

The following lemma is proved by using the above lemma.

\begin{lemma}(\cite[Lemma 2.1]{IJK})\label{subsurface distance}
Let $X$ be as above. 
Let $[\alpha_{0}, \alpha_{1}, \dots, \alpha_{n}]$ be a path in $\mathcal{C}(S)$
such that every $\alpha_{i}$ cuts $X$. 
Then ${\rm diam}_{X}(\pi_{X}(\alpha_{0}), \pi_{X}(\alpha_{n}))\leq 2n$.
\end{lemma}

Throughout this paper, given an embedding $\varphi:X\rightarrow Y$ between compact surfaces $X$ and $Y$, we abuse notation and use $\varphi$ to denote the map $\mathcal{C}^0(X)\rightarrow \mathcal{C}^0(Y)$ or $ \mathcal{P}(\mathcal{C}^0(X))\rightarrow  \mathcal{P}(\mathcal{C}^0(Y))$ induced by $\varphi:X\rightarrow Y$.
The following two lemmas can be proved by using arguments in the proof of \cite[Propositions 4.1, 4.4]{IJK}.

\begin{lemma}\label{extending geodesic}
Let $[\alpha_{0}, \alpha_{1}, \dots, \alpha_{n}]$ and $[\beta_{0}, \beta_{1}, \dots, \beta_{m}]$ be geodesics in $\mathcal{C}(S)$.
Suppose that $\alpha_{n}$ and $\beta_{0}$ are non-separating on $S$, and let $X={\rm Cl}(S\setminus N(\alpha_{n}))$.
Let $h:S\rightarrow S$ be a homeomorphism such that
\begin{itemize}
\item $h(\beta_0)=\alpha_n$, and
\item ${\rm diam}_X(\pi_X(\alpha_0), \pi_X(h(\beta_m)))>2(n+m)$.
\end{itemize}
Then $[\alpha_{0}, \alpha_{1}, \dots, \alpha_{n}(=h(\beta_{0})), h(\beta_{1}), \dots, h(\beta_{m})]$ is a geodesic in $\mathcal{C}(S)$.
\end{lemma}

\begin{lemma}\label{extending geodesic2}
Let $[\alpha_{0}, \alpha_{1}, \dots, \alpha_{n}]$ and $[\beta_{0}, \beta_{1}, \dots, \beta_{m}]$ be geodesics in $\mathcal{C}(S)$.
Suppose that $\alpha_{n-1}\cup \alpha_n$ and $\beta_0\cup \beta_1$ are non-separating on $S$, and let $X={\rm Cl}(S\setminus N(\alpha_{n-1}\cup\alpha_n))$.
Let $h:S\rightarrow S$ be a homeomorphism such that
\begin{itemize}
\item $h(\beta_0)=\alpha_{n-1}$, $h(\beta_1)=\alpha_n$, and
\item ${\rm diam}_X(\pi_X(\alpha_0), \pi_X(h(\beta_m)))>2(n+m-1)$.
\end{itemize}
Then $[\alpha_{0}, \alpha_{1}, \dots, \alpha_{n-1}(=h(\beta_{0})), \alpha_{n}(=h(\beta_{1})), h(\beta_{2}), \dots, h(\beta_{m})]$ is a geodesic in $\mathcal{C}(S)$.
\end{lemma}

\begin{remark}\label{rmk-geodesic}
{\rm
By the proof of \cite[Propositions 4.1, 4.4]{IJK}, the following holds.
\begin{itemize}
\item 
Let $[\alpha_{0}, \alpha_{1}, \dots, \alpha_{n}(=h(\beta_{0})), h(\beta_{1}), \dots, h(\beta_{m})]$ be a geodesic in Lemma~\ref{extending geodesic}.
Then every geodesic connecting $\alpha_0$ and $h(\beta_{m})$ passes through $\alpha_n$.
In fact, for any geodesic $[\gamma_0, \gamma_1, \dots, \gamma_{n+m}]$ in $\mathcal{C}(S)$ such that $\gamma_0=\alpha_0$ and $\gamma_{n+m}=h(\beta_m)$, we have $\gamma_n=\alpha_n$.
\item 
Let $[\alpha_{0}, \alpha_{1}, \dots, \alpha_{n-1}(=h(\beta_{0})), \alpha_n(=h(\beta_{1})), \dots, h(\beta_{m})]$ be a geodesic in Lemma~\ref{extending geodesic2}.
Then every geodesic connecting $\alpha_0$ and $h(\beta_{m})$ passes through $\alpha_{n-1}$ or $\alpha_n$.
In fact, for any geodesic $[\gamma_0, \gamma_1, \dots, \gamma_{n+m-1}]$ in $\mathcal{C}(S)$ such that $\gamma_0=\alpha_0$ and $\gamma_{n+m-1}=h(\beta_m)$, we have $\gamma_{n-1}=\alpha_{n-1}$ or $\gamma_n=\alpha_n$.
\end{itemize}
}
\end{remark}

\section{Keen Heegaard splittings}\label{sec-keen}

Recall that a Heegaard splitting $V_{1}\cup_S V_{2}$ is called {\it keen} if its Hempel distance is realized by a unique pair of elements of $\mathcal{D}(V_{1})$ and $\mathcal{D}(V_{2})$.

\begin{proposition}\label{prop-keen}
Let $V_1\cup_S V_2$ be a genus-$g(\geq 2)$ Heegaard splitting with Hempel distance $n(\geq 1)$.
Let $[l_0, l_1,\dots,l_n]$ be a geodesic in $\mathcal{C}(S)$ such that $l_0\in\mathcal{D}(V_1)$ and $l_n\in\mathcal{D}(V_2)$.
If $V_1\cup_S V_2$ is keen, then the following holds.
\begin{itemize}
\item[(1)] $l_0$ and $l_n$ are non-separating on $S$.
\item[(2)] $l_1$ and $l_{n-1}$ are non-separating on $S$.
\item[(3)] $l_0\cup l_1$ and $l_{n-1}\cup l_n$ are separating on $S$.
\end{itemize}
\end{proposition}

\begin{proof}
(1) Assume on the contrary that either $l_0$ or $l_n$ is separating on $S$.
Without loss of generality, we may assume that $l_0$ is separating on $S$.
Let $D_0$ be a disk properly embedded in $V_1$ such that $\partial D_0=l_0$.
Let $V_1^{(1)}$ be the component of $V_1\setminus D_0$ that contains $l_1$, and let $V_1^{(2)}$ be the other component.
It is easy to see that there is an essential disk $D_0'$ properly embedded in $V_1^{(2)}$.
Then $l_0':=\partial D_0'$ is also disjoint from $l_1$, and hence, $[l_0', l_1,\dots,l_n]$ is a geodesic in $\mathcal{C}(S)$.
Hence, we have $d_S(l_0',l_n)=d_S(\mathcal{D}(V_1),\mathcal{D}(V_2))$, where $l_0'$ is an element of $\mathcal{D}(V_1)$ different from $l_0$, a contradiction.

(2) Assume on the contrary that either $l_1$ or $l_{n-1}$, say $l_1$, is separating on $S$.
Let $S^{(1)}$ be the component of $S\setminus l_1$ that contains $l_0$.
Since $l_0$ is non-separating on $S$ by (1) and $l_1$ is separating on $S$, we can see that $l_0$ is non-separating on $S^{(1)}$.
Then there exists an essential simple closed curve $l^{\ast}$ on $S^{(1)}$ such that $l^{\ast}$ intersects $l_0$ transversely in one point.
Let $D_0$ be a disk properly embedded in $V_1$ such that $\partial D_0=l_0$, and let $D_0^+$ and $D_0^-$ be the components of ${\rm Cl}(\partial N(D_0)\setminus\partial V_1)$, where $N(D_0)$ is a regular neighborhood of $D_0$ in $V_1$. 
Take the subarc of $l^{\ast}$ lying outside of the product region $N(D_0)$ between $D_0^+$ and $D_0^-$, and let $D_0''$ be the disk in $V_1$ obtained from $D_0^+\cup D_0^-$ by adding a band along the subarc of $l^{\ast}$.
Then $l_0'':=\partial D_0''$ is  also disjoint from $l_1$, and hence, $[l_0'', l_1,\dots,l_n]$ is a geodesic in $\mathcal{C}(S)$.
Hence, we have $d_S(l_0'',l_n)=d_S(\mathcal{D}(V_1),\mathcal{D}(V_2))$, where $l_0''$ is an element of $\mathcal{D}(V_1)$ different from $l_0$, a contradiction.

(3) Assume on the contrary that either $l_0\cup l_1$ or $l_{n-1}\cup l_n$, say $l_0\cup l_1$, is non-separating on $S$.
Then there exists an essential simple closed curve $l^{\ast}$ on $S$ such that $l^{\ast}$ intersects $l_0$ transversely in one point and $l^{\ast}\cap l_1=\emptyset$.
We can lead to a contradiction by the arguments in (2).
\end{proof}

\begin{remark}\label{rmk-keen}
{\rm 
(1) By Proposition \ref{prop-keen}, we see that every genus-$2$ Heegaard splitting is not keen. 
In fact, if a genus-$2$ Heegaard splitting $V_{1}\cup_{S} V_{2}$ is keen, and $[l_{0}, l_{1}, \dots, l_{n}]$ is a path that realizes the Hempel distance, then by (1) and (2) of Proposition \ref{prop-keen}, we see that $l_{0}\cup l_{1}$ cuts $S$ into four punctured sphere, contradicting (3) of Proposition \ref{prop-keen}. 
Hence, if a genus-$g$ Heegaard splitting (with Hempel distance $n\geq 1$) is keen, then $g\geq 3$. 

(2) Heegaard splittings given in \cite{IJK, Jo, QZG} are not keen, since their Hempel distances are realized by pairs of separating elements.
}
\end{remark}

\section{Proof of Theorem~\ref{thm-1} when $n\geq 4$}\label{proof1}

Let $n$ and $g$ be integers with $n\geq 4$ and $g\geq 3$.
Let $S$ be a closed connected orientable surface of genus $g$. 
Let $l_0$ and $l_1$ be non-separating simple closed curves on $S$ such that $l_0\cap l_1=\emptyset$, $l_0\cup l_1$ is separating and $l_{0}$, $l_{1}$ are not parallel on $S$.
Let $F_1={\rm Cl}(S\setminus N(l_1))$.
Choose and fix an integer $k\in\{2,3,\dots,n-2\}$.
Let $[l_1',l_2',\dots,l_k']$ and $[l_1'',l_2'',\dots,l_{n-k}'']$ be geodesics in $\mathcal{C}(S)$ such that $l_1', l_k', l_1''$ and $l_{n-k}''$ are non-separating on $S$.
(For the existence of such geodesics, see \cite{IJK} or the proof of Claim \ref{sequence} below for example.)
By \cite[Proposition 4.6]{MM1}, there exist homeomorphisms $h_1:S\rightarrow S$ and $h_2:S\rightarrow S$ such that 
\begin{itemize}
\item $h_1(l_1')=l_1$,
\item $h_2(l_1'')=l_1$,
\item ${\rm diam}_{F_1}(\pi_{F_1}(l_0),\pi_{F_1}(h_1(l_k')))\geq 4n+16$, and 
\item ${\rm diam}_{F_1}(\pi_{F_1}(l_0),\pi_{F_1}(h_2(l_{n-k}'')))\geq 4n+16$.
\end{itemize}
Note that $\pi_{F_1}(l_0)=\{l_0\}$ since $l_0\cap l_1=\emptyset$. 
By Lemma \ref{extending geodesic}, $[l_0, l_1(=h_1(l_1')), h_1(l_2'), \dots, h_1(l_k')]$ and $[l_0, l_1(=h_2(l_1'')), h_2(l_2''), \dots, h_2(l_{n-k}'')]$ are geodesics in $\mathcal{C}(S)$.
Let $F_k={\rm Cl}(S\setminus N(h_1(l_k')))$. 
By \cite[Proposition 4.6]{MM1}, there exists a homeomorphism $h_3:S\rightarrow S$ such that
\begin{itemize}
\item $h_3(h_2(l_{n-k}''))=h_1(l_k')$, and 
\item ${\rm diam}_{F_k}(\pi_{F_k}(l_0),\pi_{F_k}(h_3(l_0)))>2n$.
\end{itemize}
Let $l_i=h_1(l_i')$ for $i\in\{2, \dots, k\}$, $l_i=h_3(h_2(l_{n-i}''))$ for $i\in\{k+1,\dots,n-1\}$, and $l_n=h_3(l_0)$.
By Lemma \ref{extending geodesic}, $[l_0, l_1,\dots, l_n]$ is a geodesic in $\mathcal{C}(S)$.
Moreover, by the construction of the geodesic, the following are satisfied.
\begin{itemize}
\item[(G1)] $l_0, l_1, l_{n-1}$ and $l_n$ are non-separating on $S$,
\item[(G2)] $l_0\cup l_1$ and $l_{n-1}\cup l_n$ are separating on $S$,
\item[(G3)] ${\rm diam}_{F_1}(\pi_{F_1}(l_0), \pi_{F_1}(l_k))\geq 4n+16$, 
\item[(G4)] ${\rm diam}_{F_{n-1}}(\pi_{F_{n-1}}(l_k), \pi_{F_{n-1}}(l_n))\geq 4n+16$, and
\item[(G5)] ${\rm diam}_{F_k}(\pi_{F_k}(l_0), \pi_{F_k}(l_n))>2n$,
\end{itemize}
where $F_{n-1}={\rm Cl}(S\setminus N(l_{n-1}))$.

Let $C_1$ and $C_2$ be copies of the compression-body obtained by adding a $1$-handle to $F\times [0,1]$, where $F$ is a closed orientable surface of genus $g-1$.
Let $D_1$ (resp. $D_2$) be the non-separating essential disk properly embedded in $C_1$ (resp. $C_2$) corresponding to  the co-core of the 1-handle.
We may assume that $\partial_+C_1=S$ and $\partial D_1=l_0$.
Choose a homeomorphism $f:\partial_+C_2\rightarrow \partial_+C_1$ such that $f(\partial D_2)=l_n$.

Let $H_1$ and $H_2$ be copies of the handlebody of genus $g-1$.
In the remainder of this section, we identify $\partial H_i$ and $\partial_- C_i$ ($i=1,2$) so that we obtain a keen Heegaard splitting of genus $g$ whose Hempel distance is $n$.

For each $i=1,2$, let $C_i'={\rm Cl}(C_i\setminus N(D_i))$ and $X_i=\partial C_i'\cap \partial_+ C_i$. 
Note that $C_i'$ is homeomorphic to $\partial_- C_i\times [0,1]$.
Let $\varphi_i:C_i'\rightarrow \partial_- C_i\times [0,1]$ be a homeomorphism such that $\varphi_i(\partial C_i'\setminus \partial_- C_i)=\partial_- C_i\times\{1\}$ and $\varphi_i(\partial_- C_i)=\partial_- C_i\times\{0\}$, and let $\psi_i:\partial_- C_i\times\{1\}\rightarrow \partial_- C_i\times\{0\}$ be the natural homeomorphism.
Let $P_i:X_i\rightarrow \partial_- C_i$ be the composition of the inclusion map $X_i\rightarrow \partial C_i'\setminus \partial_- C_i$ and the map $\left(\varphi_i|_{\partial_-C_i}\right)^{-1}\circ \psi_i \circ \left(\varphi_i|_{\partial C_i'\setminus \partial_- C_i}\right): \partial C_i'\setminus \partial_- C_i\rightarrow \partial_-C_i$.

It is clear that $l_1$ represents an essential simple closed curve on $X_1$.
Since $l_1$ is non-separating on $S$, $P_1(l_1)$ is an essential simple closed curve on $\partial_- C_1$.
By \cite{AS}, there exists a homeomorphism $f_1:\partial H_1 \rightarrow \partial_- C_1$ such that 
\begin{equation}\label{eqn-f1}
d_{\partial_- C_1} (f_1(\mathcal{D}(H_1)), P_1(l_1))\geq 2.
\end{equation}
Let $V_1=C_1\cup_{f_1} H_1$, that is, $V_1$ is the manifold obtained from $C_1$ and $H_1$ by identifying $\partial_- C_1$ and $\partial H_1$ via $f_1$. Note that $V_{1}$ is a handlebody.

\begin{claim}\label{claim-l1-v1}
$l_1$ intersects every element of $\mathcal{D}(V_1)\setminus \{l_0\}$.
\end{claim}

\begin{proof}
Assume on the contrary that there exists an element $a$ of $\mathcal{D}(V_1)\setminus \{l_0\}$ such that $a\cap l_1=\emptyset$.
Let $D_a$ be a disk in $V_1$ bounded by $a$, and recall that $l_0$ bounds the disk $D_1$ in $C_1$, and hence, in $V_1$  (see Fig.~\ref{fig:1}).
We may assume that $|D_a\cap D_1|=|D_a\cap N(D_1)|$ and is minimal.
By using innermost disk arguments, we see that $D_a\cap D_1$ has no loop components.
Let $\Delta$ be a disk properly embedded in $C_1'\cup_{f_1} H_1$ defined as follows.
\begin{itemize}
\item If $D_a\cap D_1=\emptyset$, let $\Delta=D_a$.
\item If $D_a\cap D_1\ne\emptyset$, let $\Delta$ be the closure of a component of $D_a\setminus N(D_1)$ that is outermost in $D_a$.
\end{itemize}
Since $a\cap l_1=\emptyset$, the disk $\Delta$ is disjoint from $l_1$.
Since $l_{0}\cup l_{1}$ is separating on $S$ by the condition (G2), and $a\neq l_{0}$, we see that $\Delta$ is essential in $C_1'\cup_{f_1} H_1$.  

\begin{figure}[htbp]
 \begin{center}
 \includegraphics[width=55mm]{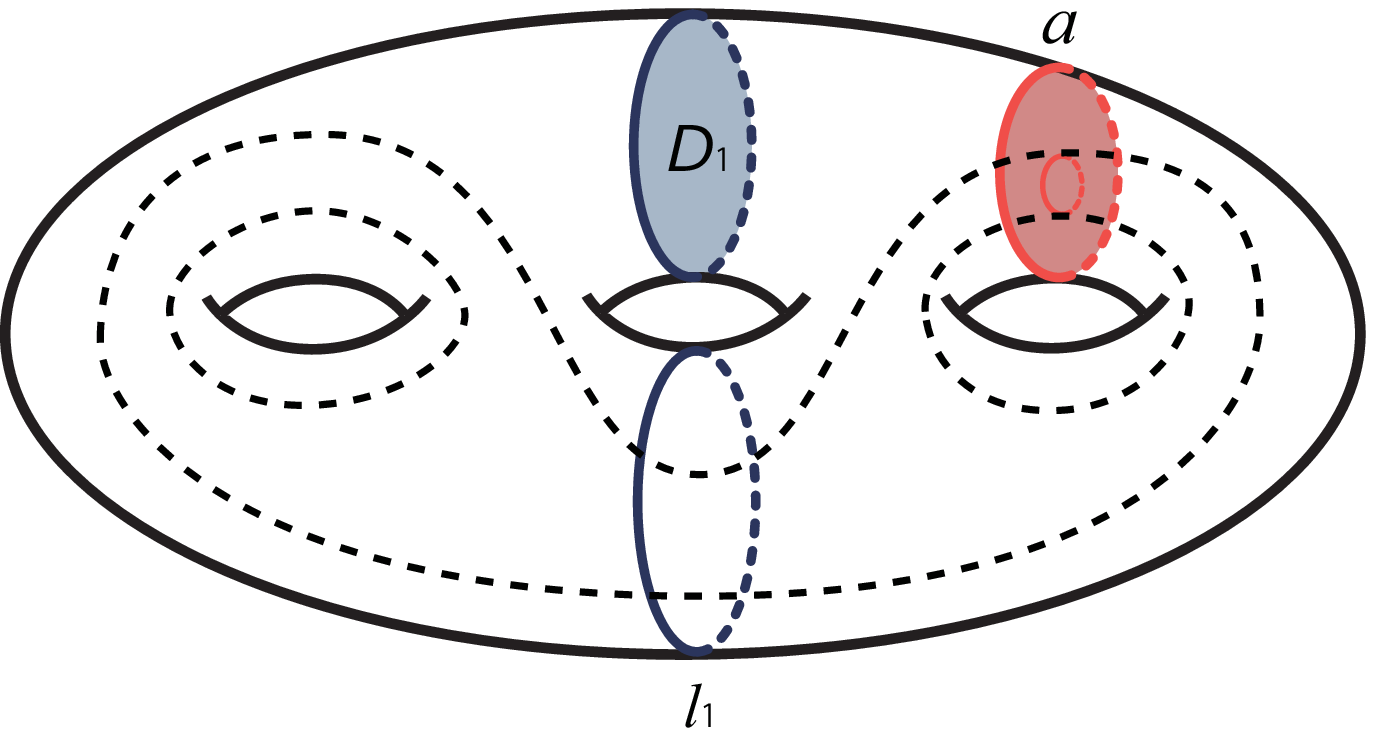}
 \end{center}
 \caption{}
\label{fig:1}
\end{figure}
Since $C_1'$ is homeomorphic to $\partial_- C_1\times [0,1]$, we may assume that $\Delta$ is obtained by gluing a vertical annulus in $C_1'$ and an essential disk $\Delta'$ in $H_1$ via $f_1$, after boundary compressions and isotopies toward $\partial_- C_1$ if necessary.
This together with $\Delta\cap l_1=\emptyset$ implies that $d_{\partial_- C_1}(f_1(\partial \Delta'), P_1(l_1))\leq 1$.
Since $f_1(\partial \Delta')\in f_1(\mathcal{D}(H_1))$, we have $d_{\partial_- C_1}(f_1(\mathcal{D}(H_1)), P_1(l_1))\leq 1$, a contradiction to the inequality (\ref{eqn-f1}).
\end{proof}

Let $\pi_{F_1}=\pi_0\circ\pi_A:\mathcal{C}^{0}(S)\rightarrow \mathcal{P}(\mathcal{AC}^{0}(F_1))\rightarrow \mathcal{P}(\mathcal{C}^{0}(F_1))$ be the subsurface projection introduced in Section~\ref{sec-pre}.
Recall that $\pi_{F_1}(l_0)=\{l_0\}$ since $l_0\cap l_1=\emptyset$.

\begin{claim}\label{claim-l0-v1}
For any element $a\in \mathcal{D}(V_1)$, we have $\pi_{F_1}(a)\neq \emptyset$, and ${\rm diam}_{F_1}(l_0, \pi_{F_1}(a))\leq 4$. 
\end{claim}

\begin{proof}
Note that, by Claim \ref{claim-l1-v1}, we immediately have $\pi_{F_1}(a)\neq \emptyset$. 
If $a=l_0$ or $a\cap l_0=\emptyset$, that is, $d_S(l_0,a)\leq 1$, then we have ${\rm diam}_{F_1}(l_0, \pi_{F_1}(a))\leq 2$ by Lemma~\ref{subsurface distance}.
Hence, we suppose that $a\ne l_0$ and $a\cap l_0\ne\emptyset$ in the following.

Let $D_a$ be a disk in $V_1$ bounded by $a$, and recall that $l_0$ bounds the disk $D_1$ in $V_1$.
Here, we may assume that $|a\cap l_1|=|a\cap N(l_1)|$ and is minimal. 
We may also assume that $|D_a\cap D_1|=|D_a\cap N(D_1)|$ and is minimal.
Let $\Delta$ be the closure of a component of $D_a\setminus N(D_1)$ that is outermost in $D_a$.
If $\Delta\cap l_1=\emptyset$, then we can lead to a contradiction by arguments in the proof of Claim~\ref{claim-l1-v1}.
Hence, $\Delta\cap l_1\ne\emptyset$.
Since $l_0\cup l_1$ is separating on $S$ by the condition (G2), there exists a component $\gamma$ of ${\rm Cl}(\partial\Delta \setminus (N(D_1)\cup N(l_1)))$ such that $\partial \gamma \subset \partial N(l_1)$.
It is clear that $\gamma$ is an essential arc on $F_1$.
Note that $\gamma$ is disjoint from $l_0$, that is, $d_{\mathcal{AC}(F_1)}(l_0,\gamma)=1$, since $l_0\cap \Delta=\emptyset$ and $\gamma$ is a subarc of $\partial \Delta$.
Since $\gamma\in \pi_A(a)$, we have $d_{\mathcal{AC}(F_1)}(l_0,\pi_A(a))\leq d_{\mathcal{AC}(F_1)}(l_0,\gamma)=1$.
Hence, 
$${\rm diam}_{\mathcal{AC}(F_1)}(l_0,\pi_A(a))\leq d_{\mathcal{AC}(F_1)}(l_0,\pi_A(a))+{\rm diam}_{\mathcal{AC}(F_1)}(\pi_A(a))\leq 1+1=2.$$
By Lemma~\ref{lem-ac-to-c}, we have ${\rm diam}_{F_1}(l_0,\pi_{F_1}(a))\leq 4$.
\end{proof}

\begin{lemma}\label{lem-h1}
$d_{S} (\mathcal{D}(V_1), l_n)=n$.
\end{lemma}

\begin{proof}
Since $l_0\in \mathcal{D}(V_1)$, we have $d_{S} (\mathcal{D}(V_1), l_n)\leq n$.
To prove $d_{S} (\mathcal{D}(V_1), l_n)=n$, assume on the contrary that $d_{S} (\mathcal{D}(V_1), l_n)<n$.
Then there exists a geodesic $[m_0,m_1,\dots,m_p]$ in $\mathcal{C}(S)$ such that $p<n$, $m_0\in \mathcal{D}(V_1)$ and $m_p=l_n$.

\begin{claim}\label{claim-1}
$m_i=l_1$ for some $i\in\{0,1,\dots,p\}$.
\end{claim}

\begin{proof}
Assume on the contrary that $m_i\neq l_1$ for every $i\in\{0,1,\dots,p\}$.
Namely, every $m_i$ cuts $F_1$.
By Lemma \ref{subsurface distance}, we have
\begin{equation}\label{claim-1-eqn-1}
{\rm diam}_{F_1}(\pi_{F_1}(m_0), \pi_{F_1}(m_p))\leq 2p.
\end{equation}
Similarly, we have 
\begin{equation}\label{claim-1-eqn-2}
{\rm diam}_{F_1}(\pi_{F_1}(l_n), \pi_{F_1}(l_k))\leq 2(n-k).
\end{equation}
By the triangle inequality, we have
\begin{eqnarray}\label{claim-1-eqn-3}
\begin{array}{rcl}
{\rm diam}_{F_1}(\pi_{F_1}(l_0), \pi_{F_1}(l_k)) &\leq& {\rm diam}_{F_1}(\pi_{F_1}(l_0), \pi_{F_1}(m_0))
\\&&+{\rm diam}_{F_1}(\pi_{F_1}(m_0), \pi_{F_1}(m_p))\\&&+{\rm diam}_{F_1}(\pi_{F_1}(l_n), \pi_{F_1}(l_k)).
\end{array}
\end{eqnarray}
By the inequalities (\ref{claim-1-eqn-1}), (\ref{claim-1-eqn-2}), (\ref{claim-1-eqn-3}) and Claim~\ref{claim-l0-v1}, we obtain
\begin{eqnarray}\label{claim-1-eqn-4}
\begin{array}{rcl}
{\rm diam}_{F_1}(\pi_{F_1}(l_0), \pi_{F_1}(l_k)) &\leq& 4+2p+2(n-k)
\\&<& 4+2n+2n,
\end{array}
\end{eqnarray}
which contradicts the condition (G3).
\end{proof}

By Claim~\ref{claim-1}, we have $d_S(m_i,m_p)=d_S(l_1,l_n)$.
Since $[m_0,m_1,\dots,m_p]$ and $[l_0,l_1,\dots,l_n]$ are geodesics, $d_S(m_i,m_p)=p-i$ and $d_S(l_1,l_n)=n-1>p-1$.
Hence, $p-i>p-1$, which implies $i=0$, that is, $m_0=l_1$.
This contradicts Claim~\ref{claim-l1-v1}.
Hence, we have $d_{S} (\mathcal{D}(V_1), l_n)=n$.
\end{proof}

Note that $f^{-1}(l_{n-1})$ represents an essential simple closed curve on $X_2$.
Since $f^{-1}(l_{n-1})$ is non-separating on $\partial_+ C_2$ by the condition (G1), $P_2(f^{-1}(l_{n-1}))$ is an essential simple closed curve on $\partial_- C_2$.
By \cite{AS}, there exists a homeomorphism $f_2:\partial H_2 \rightarrow \partial_- C_2$ such that 
\begin{equation}\label{eqn-f2}
d_{\partial_- C_2} (f_2(\mathcal{D}(H_2)), P_2(f^{-1}(l_{n-1})))\geq 2.
\end{equation}
Let $V_2=C_2\cup_{f_2} H_2$. 
Then $V_1\cup_f V_2$ is a genus-$g$ Heegaard splitting.

Claims~\ref{claim-ln-1-v2}, \ref{claim-ln-v2} and Lemma \ref{lem-h2} below can be proved by the arguments similar to those for Claims~\ref{claim-l1-v1}, \ref{claim-l0-v1} and Lemma \ref{lem-h1}, respectively.

\begin{claim}\label{claim-ln-1-v2}
$l_{n-1}$ intersects every element of $f(\mathcal{D}(V_2))\setminus \{l_n\}$.
\end{claim}

\begin{claim}\label{claim-ln-v2}
For any element $a\in f(\mathcal{D}(V_2))$, we have $\pi_{F_{n-1}}(a)\neq \emptyset$, and ${\rm diam}_{F_{n-1}}(l_n, \pi_{F_{n-1}}(a))\leq 4$.
\end{claim}

\begin{lemma}\label{lem-h2}
$d_{S} (f(\mathcal{D}(V_2)), l_0)=n$.
\end{lemma}

\begin{claim}\label{claim-l1-v2}
${\rm diam}_{F_1}(\pi_{F_1}(f(\mathcal{D}(V_2))))\leq 12$ and ${\rm diam}_{F_{n-1}}(\pi_{F_{n-1}}(\mathcal{D}(V_1)))\leq 12$.
\end{claim}

\begin{proof}
By Lemma~\ref{lem-h1}, we have $d_S(\mathcal{D}(V_1),l_{n-1})=n-1\geq 3$.
Hence, by \cite[Theorem 1]{Li}, ${\rm diam}_{F_{n-1}}(\pi_{F_{n-1}}(\mathcal{D}(V_1)))\leq 12$. 
Similarly, we have ${\rm diam}_{F_1}(\pi_{F_1}(f(\mathcal{D}(V_2))))\leq 12$ by Lemma~\ref{lem-h2} and \cite{Li}.
\end{proof}

\begin{lemma}\label{lem-distance}
$d_{S} (\mathcal{D}(V_1), f(\mathcal{D}(V_2)))=n$. Namely, the Hempel distance of the Heegaard splitting $V_1\cup_f V_2$ is $n$.
\end{lemma}

\begin{proof}
Since $l_0\in \mathcal{D}(V_1)$ and $l_n\in f(\mathcal{D}(V_2))$, we have $d_{S} (\mathcal{D}(V_1), f(\mathcal{D}(V_2)))\leq n$.
Let $[m_0,m_1,\dots,m_p]$ be a geodesic in $\mathcal{C}(S)$ such that $m_0\in \mathcal{D}(V_1)$, $m_p\in f(\mathcal{D}(V_2))$ and $p\leq n$.

\begin{claim}\label{claim-2}
$m_i=l_1$ for some $i\in\{0,1,\dots,p\}$.
\end{claim}

\begin{proof}
Assume on the contrary that $m_i\neq l_1$ for every $i\in\{0,1,\dots,p\}$.
Namely, every $m_i$ cuts $F_1$.
By Lemma \ref{subsurface distance}, we have
\begin{equation}\label{claim-2-eqn-1}
{\rm diam}_{F_1}(\pi_{F_1}(m_0), \pi_{F_1}(m_p))\leq 2p.
\end{equation}
Recall that $k\in\{2,3,\dots,n-2\}$. 
Similarly, we have 
\begin{equation}\label{claim-2-eqn-2}
{\rm diam}_{F_1}(\pi_{F_1}(l_n), \pi_{F_1}(l_k))\leq 2(n-k).
\end{equation}
By the triangle inequality, we have
\begin{eqnarray}\label{claim-2-eqn-3}
\begin{array}{rcl}
{\rm diam}_{F_1}(\pi_{F_1}(l_0), \pi_{F_1}(l_k)) &\leq& {\rm diam}_{F_1}(\pi_{F_1}(l_0), \pi_{F_1}(m_0))
\\&&+{\rm diam}_{F_1}(\pi_{F_1}(m_0), \pi_{F_1}(m_p))
\\&&+{\rm diam}_{F_1}(\pi_{F_1}(m_p), \pi_{F_1}(l_n))
\\&&+{\rm diam}_{F_1}(\pi_{F_1}(l_n), \pi_{F_1}(l_k)).
\end{array}
\end{eqnarray}
By the inequalities (\ref{claim-2-eqn-1}), (\ref{claim-2-eqn-2}), (\ref{claim-2-eqn-3}) together with Claims~\ref{claim-l0-v1} and \ref{claim-l1-v2}, we obtain
\begin{eqnarray}\label{claim-2-eqn-4}
\begin{array}{rcl}
{\rm diam}_{F_1}(\pi_{F_1}(l_0), \pi_{F_1}(l_k)) &\leq& 4+2p+12+2(n-k)
\\&<& 4+2n+12+2n,
\end{array}
\end{eqnarray}
which contradicts the condition (G3).
\end{proof}

The following claim can be proved similarly.

\begin{claim}\label{claim-3}
$m_j=l_{n-1}$ for some $j\in\{0,1,\dots,p\}$.
\end{claim}

Note that $l_1\not\in \mathcal{D}(V_1)$ by Claim~\ref{claim-l1-v1}.
Note also that $l_1\not\in f(\mathcal{D}(V_2))$ since, otherwise, we have $d_S(f(\mathcal{D}(V_2)), l_0)\leq d_S(l_1,l_0)=1$, which contradicts Lemma~\ref{lem-h2}.
Since $m_0\in \mathcal{D}(V_1)$ and $m_p\in f(\mathcal{D}(V_2))$ by the assumption, we have $m_i(=l_1)\ne m_0$ and $m_i(=l_1)\ne m_p$, which implies $1\leq i\leq p-1$.
Similarly, we have $1\leq j\leq p-1$.
Hence, we have 
\begin{equation}\label{e1}
|i-j|\leq (p-1)-1=p-2.
\end{equation}
On the other hand, by Claims~\ref{claim-2} and \ref{claim-3}, we have 
$$
|i-j|=d_S(m_i,m_j)=d_S(l_1,l_{n-1})=n-2,
$$
which together with the inequality (\ref{e1}) implies $p=n$.
Hence, $d_{S} (\mathcal{D}(V_1), f(\mathcal{D}(V_2)))=n$.
\end{proof}

\begin{lemma}\label{lem-keen}
The Heegaard splitting $V_1\cup_f V_2$ is keen.
\end{lemma}

\begin{proof}
Let $[m_0,m_1,\dots,m_n]$ be a geodesic in $\mathcal{C}(S)$ such that $m_0\in \mathcal{D}(V_1)$ and $m_n\in f(\mathcal{D}(V_2))$.
By the proof of Lemma \ref{lem-distance}, we have $m_1=l_1$ and $m_{n-1}=l_{n-1}$.
By Claims~\ref{claim-l1-v1} and \ref{claim-ln-1-v2}, we have $m_0=l_0$ and $m_n=l_n$.
\end{proof}

In the remainder of this section, we show that the existence of strongly keen Heegaard splitting. 

\begin{claim}
In the above construction, if the following conditions are satisfied, then the Heegaard splitting constructed from the geodesic $[l_{0}, l_{1}, \dots, l_{n}]$ is strongly keen. 
\begin{itemize}
\item The geodesic $[l_1',l_2',\dots,l_k']$ (resp. $[l_1'',l_2'',\dots,l_{n-k}'']$) is the unique geodesic from $l_0'$ to $l_k'$ (resp. $l_0''$ to $l_{n-k}''$). 
\end{itemize}

\end{claim}

\begin{proof}
By the proof of Lemma~\ref{lem-keen}, $m_i=l_i$ holds for $i=0,1,n-1$ and $n$.
Moreover, by the condition (G5) and Remark~\ref{rmk-geodesic}, we have $m_k=l_k$.
Hence, if the geodesics $[l_1',l_2',\dots,l_k']$ (resp. $[l_1'',l_2'',\dots,l_{n-k}'']$) is the unique geodesic containing $l_0'$ and $l_k'$ (resp. $l_0''$ and $l_{n-k}''$), then we obtain the desired result.
\end{proof}


Hence the next claim completes the proof of Theorem \ref{thm-1}. 

\begin{claim}\label{sequence}
For each $p$, there exists a geodesic $[\alpha_0,\alpha_1,\dots, \alpha_p]$ such that each $\alpha_i$ $(i=0,1,\dots,p)$ is non-separating on $S$ and $[\alpha_0,\alpha_1,\dots, \alpha_p]$ is the unique geodesic connecting $\alpha_0$ and $\alpha_p$.
\end{claim}

\begin{proof}
Let $\alpha_0$ and $\alpha_1$ be non-separating simple closed curve on $S$ such that $\alpha_0\cap\alpha_1=\emptyset$, and let $X_1={\rm Cl}(S\setminus N(\alpha_1))$.
Let $\alpha_2'$ be a non-separating simple closed curve on $S$ disjoint from $\alpha_1$. 
By \cite[Proposition 4.6]{MM1}, there exists a homeomorphism $g_1:S\rightarrow S$ such that $g_1(\alpha_1)=\alpha_1$ and ${\rm diam}_{X_1}(\pi_{X_1}(\alpha_0), \pi_{X_1}(g_1(\alpha_2')))>4$. 
Let $\alpha_2=g_1(\alpha_2')$. 
By Lemma~\ref{extending geodesic}, $[\alpha_0,\alpha_1,\alpha_2]$ is a geodesic in $\mathcal{C}(S)$.
Moreover, by Remark~\ref{rmk-geodesic}, $[\alpha_0,\alpha_1,\alpha_2]$ is the unique geodesic connecting $\alpha_0$ and $\alpha_2$.

For any positive integer $p$, we repeat this process to construct a geodesic $[\alpha_0,\alpha_1,\dots, \alpha_p]$ inductively as follows.
Suppose we have constructed a geodesic $[\alpha_0,\alpha_1,\dots, \alpha_i]$ for $i<p$ such that 
\begin{itemize}
\item $\alpha_i$ is non-separating on $S$, and
\item $[\alpha_0,\alpha_1,\dots, \alpha_i]$ is the unique geodesic connecting $\alpha_0$ and $\alpha_i$.
\end{itemize}
Let $X_i={\rm Cl}(S\setminus N(\alpha_i))$.
Let $\alpha_{i+1}'$ be a non-separating simple closed curve on $S$ disjoint from $\alpha_i$. 
By \cite[Proposition 4.6]{MM1}, there exists a homeomorphism $g_i:S\rightarrow S$ such that $g_i(\alpha_i)=\alpha_i$ and ${\rm diam}_{X_i}(\pi_{X_i}(\alpha_0), \pi_{X_i}(g_i(\alpha_{i+1}')))>2(i+1)$. 
Let $\alpha_{i+1}=g_i(\alpha_{i+1}')$. 
By Lemma~\ref{extending geodesic}, $[\alpha_0,\alpha_1,\dots, \alpha_{i+1}]$ is a geodesic in $\mathcal{C}(S)$.
Moreover, by Remark~\ref{rmk-geodesic}, every geodesic connecting $\alpha_0$ and $\alpha_{i+1}$ passes through $\alpha_i$.
Since $[\alpha_0,\alpha_1,\dots, \alpha_i]$ is the unique geodesic connecting $\alpha_0$ and $\alpha_i$, we have that $[\alpha_0,\alpha_1,\dots, \alpha_{i+1}]$ is the unique geodesic connecting $\alpha_0$ and $\alpha_{i+1}$.
Hence, we obtain a geodesic $[\alpha_0,\alpha_1,\dots, \alpha_p]$ such that every $\alpha_i$ $(i=0,1,\dots,p)$ is non-separating on $S$ and $[\alpha_0,\alpha_1,\dots, \alpha_p]$ is the unique geodesic connecting $\alpha_0$ and $\alpha_p$.
\end{proof}

\section{Proof of Theorem~\ref{thm-1} when $n=2$}\label{proof2}

Let $n=2$ and $g$ be an integer with $g\geq 3$.
Let $S$ be a closed connected orientable surface of genus $g$.
Let $l_0$ and $l_1$ be non-separating simple closed curves on $S$ such that $l_0\cup l_1$ is separating on $S$ and $l_{0}$, $l_{1}$ are not parallel on $S$. 
By \cite[Proposition 4.6]{MM1}, there exists a homeomorphism $h:S\rightarrow S$ such that $h(l_1)=l_1$ and 
$$d_{F_1}(l_0, h(l_0))>12,$$ 
where $F_1={\rm Cl}(S\setminus N(l_1))$. 
Let $l_2=h(l_0)$.
By Lemma~\ref{extending geodesic}, $[l_0,l_1,l_2]$ is a geodesic in $\mathcal{C}(S)$.

Let $C_1$ and $C_2$ be copies of the compression-body obtained by adding a $1$-handle to $F\times [0,1]$, where $F$ is a closed orientable surface of genus $g-1$.
Let $D_1$ and $D_2$ be the non-separating essential disk properly embedded in $C_1$ and $C_2$ corresponding to the co-cores of the 1-handles, respectively.
We may assume that $\partial_+C_1=S$ and $\partial D_1=l_0$.
Choose a homeomorphism $f:\partial_+C_2\rightarrow \partial_+C_1$ such that $f(\partial D_2)=l_2$.

Let $H_i, C_i', X_i, P_i$ ($i=1,2$) be as in Section~\ref{proof1}.
Note that $l_1$ is non-separating on $S$, and hence, $P_1(l_1)$ and $P_2(f^{-1}(l_1))$ are essential simple closed curves on $\partial_- C_1$ and $\partial_- C_2$, respectively.
By \cite{AS}, there exist homeomorphisms $f_1:\partial H_1\rightarrow \partial_- C_1$ and $f_2:\partial H_2\rightarrow \partial_- C_2$ such that $d_{\partial_- C_1} (f_1(\mathcal{D}(H_1)), P_1(l_1))\geq 2$ and $d_{\partial_- C_2} (f_2(\mathcal{D}(H_2)), P_2(f^{-1}(l_1)))\geq 2$, respectively.
Let $V_i=C_i\cup_{f_i} H_i$ ($i=1,2$).
Then, $V_1\cup_f V_2$ is a genus-$g$ Heegaard splitting.
By the arguments similar to those for Claims~\ref{claim-l1-v1}, \ref{claim-l0-v1}, \ref{claim-ln-1-v2} and \ref{claim-ln-v2}, we obtain the following. 

\begin{claim}\label{claim-n2}
{\rm (1)} $l_1$ intersects every element of $\mathcal{D}(V_1)\setminus \{l_0\}$ and every element of $f(\mathcal{D}(V_2))\setminus \{l_2\}$.

{\rm (2)} For any element $a\in\mathcal{D}(V_1)$, we have $\pi_{F_1}(a)\neq \emptyset$, and ${\rm diam}_{F_1}(l_0, \pi_{F_1}(a))\leq 4$.

{\rm (3)} For any element $a\in f(\mathcal{D}(V_2))$, we have $\pi_{F_{1}}(a)\neq \emptyset$, and ${\rm diam}_{F_1}(l_2, \pi_{F_1}(a))\leq 4$.
\end{claim}

\begin{lemma}\label{lem-n2}
$V_1\cup_f V_2$ is a strongly keen Heegaard splitting whose Hempel distance is $2$.
\end{lemma}

\begin{proof}
Since $l_0\in \mathcal{D}(V_1)$ and $l_2\in f(\mathcal{D}(V_2))$, we have $d_S(\mathcal{D}(V_1), f(\mathcal{D}(V_2)))\leq 2$.
Let $[m_0, m_1, m_2]$ be a geodesic in $\mathcal{C}(S)$ such that $m_0\in \mathcal{D}(V_1)$ and $m_2\in f(\mathcal{D}(V_2))$.
(Possibly, $m_1\in \mathcal{D}(V_1)$ or $m_1\in  f(\mathcal{D}(V_2))$.)
By Claim~\ref{claim-n2} (1), both $m_0$ and $m_2$ cut $F_1$.
If $m_1$ also cuts $F_1$, then we have ${\rm diam}_{F_1}(\pi_{F_1}(m_0), \pi_{F_1}(m_2))\leq 4$ by Lemma~\ref{subsurface distance}, which together with Claim~\ref{claim-n2} (2) and (3) implies that 
\begin{eqnarray*}
\begin{array}{rcl}
d_{F_1}(l_0, l_2)&\leq& {\rm diam}_{F_1}(l_0, \pi_{F_1}(m_0))+{\rm diam}_{F_1}(\pi_{F_1}(m_0), \pi_{F_1}(m_2))\\
&&+{\rm diam}_{F_1}(\pi_{F_1}(m_2),l_2)\\
&\leq& 4+4+4=12.
\end{array}
\end{eqnarray*}
This contradicts the fact that $d_{F_1}(l_0, l_2)>12$.
Hence, $m_1$ misses $F_1$, that is, $m_1=l_1$.
By Claim~\ref{claim-n2} (1), we have $m_0=l_0$ and $m_2=l_2$, and we obtain the desired result.
\end{proof}

\section{Proof of Theorem~\ref{thm-1} when $n=3$}\label{proof3}

Let $n=3$ and $g$ be an integer with $g\geq 3$.
Let $S$ be a closed connected orientable surface of genus $g$.
Let $l_0$ and $l_1$ be non-separating simple closed curves on $S$ such that $l_0\cup l_1$ is separating on $S$ and $l_{0}$, $l_{1}$ are not parallel on $S$.
Let $l_2'$ be a simple closed curve on $S$ such that $l_2'\cap l_1=\emptyset$ and $l_1\cup l_2'$ is non-separating on $S$.
By \cite[Proposition 4.6]{MM1}, there exists a homeomorphism $h_1:S\rightarrow S$ such that $h_1(l_1)=l_1$ and $$d_{F_1}(l_0, h_1(l_2'))>8,$$ where $F_1={\rm Cl}(S\setminus N(l_1))$.
Let $l_2=h_1(l_2')$.
By Lemma~\ref{extending geodesic}, $[l_0,l_1,l_2]$ is a geodesic in $\mathcal{C}(S)$.
Note that there exists a homeomorphism $h_2:S\rightarrow S$ such that $h_2(l_1)=l_2$ and $h_2(l_2)=l_1$, since $l_1\cup l_2$ is non-separating on $S$.
Let $l_3'=h_2(l_0)$.
Note that $[l_1,l_2,l_3']$ is a geodesic in $\mathcal{C}(S)$.

Let $S'={\rm Cl}(S\setminus N(l_1\cup l_2))$.
Let $\pi_{S'}=\pi_0\circ \pi_A:\mathcal{C}^0(S)\rightarrow \mathcal{P}(\mathcal{AC}^0(S'))\rightarrow \mathcal{P}(\mathcal{C}^0(S'))$ be the subsurface projection introduced in Section~\ref{sec-pre}.

\begin{claim}\label{claim-s'}
There exists a homeomorphism $h:S\rightarrow S$ such that $h(l_1)=l_1$, $h(l_2)=l_2$ and 
${\rm diam}_{S'}(\pi_{S'}(l_0), \pi_{S'}(h(l_3')))>14$.
\end{claim}

\begin{proof}
Let $\gamma$ be the closure of a component of $l_3'\setminus l_1$.
Since $l_3'\cap l_2=\emptyset$, we have $\gamma\in\pi_{A}(l_3')$, and hence, $\pi_0(\gamma)\in\pi_0(\pi_{A}(l_3'))=\pi_{S'}(l_3')$.
Note that $\pi_0(\gamma)$ consists of a single simple closed curve or two disjoint simple closed curves on $S'$.
Since the diameter of $\mathcal{C}(S')$ is infinite, there exists a homeomorphism $h:S\rightarrow S$ such that $h(l_1)=l_1$, $h(l_2)=l_2$ and 
$d_{S'}(\pi_{S'}(l_0), h(\pi_0(\gamma)))>14$.
This inequality, together with the fact that $h(\pi_0(\gamma))\in h(\pi_{S'}(l_3'))$, implies 
\begin{eqnarray*}
\begin{array}{rcl}
{\rm diam}_{S'}(\pi_{S'}(l_0), \pi_{S'}(h(l_3')))&=&{\rm diam}_{S'}(\pi_{S'}(l_0), h(\pi_{S'}(l_3')))\\
&\geq &d_{S'}(\pi_{S'}(l_0), h(\pi_0(\gamma)))\\
&>&14.
\end{array}
\end{eqnarray*}
\end{proof}

Let $l_3=h(l_3')$.
By Lemma~\ref{extending geodesic2}, $[l_0,l_1,l_2,l_3]$ is a geodesic in $\mathcal{C}(S)$.
Note that the following hold.
\begin{itemize}
\item $d_{F_1}(l_0,l_2)>8$.
\item $d_{F_2}(l_1,l_3)>8$, where $F_2={\rm Cl}(S\setminus N(l_2))$, since $d_{F_1}(l_0,l_2)>8$ and the homeomorphism $h\circ h_2$ sends $l_0,l_1,l_2$ to $l_3,l_2,l_1$, respectively.
\item ${\rm diam}_{S'}(\pi_{S'}(l_0), \pi_{S'}(l_3))>14$.
\end{itemize}

Let $C_1$ and $C_2$ be copies of the compression-body obtained by adding a $1$-handle to $F\times [0,1]$, where $F$ is a closed orientable surface of genus $g-1$.
Let $D_1$ and $D_2$ be the non-separating essential disk properly embedded in $C_1$ and $C_2$ corresponding to the co-cores of the 1-handles, respectively.
We may assume that $\partial_+C_1=S$ and $\partial D_1=l_0$.
Choose a homeomorphism $f:\partial_+C_2\rightarrow \partial_+C_1$ such that $f(\partial D_2)=l_3$.

Let $H_i, C_i', X_i, P_i$ ($i=1,2$) be as in Section~\ref{proof1}.
Note that $l_1$ and $l_2$ are non-separating on $S$ and not isotopic to $l_{0}$ or $l_{3}$. 
Hence, $P_1(l_1)$ and $P_2(f^{-1}(l_2))$ are essential simple closed curves on $\partial_- C_1$ and $\partial_- C_2$, respectively.
By \cite{AS}, there exist homeomorphisms $f_1:\partial H_1\rightarrow \partial_- C_1$ and $f_2:\partial H_2\rightarrow \partial_- C_2$ such that 
$d_{\partial_- C_1} (f_1(\mathcal{D}(H_1)), P_1(l_1))\geq 2$ and $d_{\partial_- C_2} (f_2(\mathcal{D}(H_2)), P_2(f^{-1}(l_2)))\geq 2$, respectively.
Let $V_i=C_i\cup_{f_i} H_i$ ($i=1,2$).
Then, $V_1\cup_f V_2$ is a genus-$g$ Heegaard splitting.
By the arguments similar to those for Claims~\ref{claim-l1-v1}, \ref{claim-l0-v1}, \ref{claim-ln-1-v2} and \ref{claim-ln-v2}, we obtain the following. 

\begin{claim}\label{claim-n3-1}
{\rm (1)} $l_1$ intersects every element of $\mathcal{D}(V_1)\setminus \{l_0\}$, and $l_2$ intersects every element of $f(\mathcal{D}(V_2))\setminus \{l_3\}$.

{\rm (2)} For any element $a\in\mathcal{D}(V_1)$, we have $\pi_{F_1}(a)\neq \emptyset$, and ${\rm diam}_{F_1}(l_0, \pi_{F_1}(a))\leq 4$.

{\rm (3)} For any element $a\in f(\mathcal{D}(V_2))$, we have $\pi_{F_2}(a)\neq \emptyset$, and ${\rm diam}_{F_2}(l_3, \pi_{F_2}(a))\leq 4$.
\end{claim}

\begin{lemma}\label{claim-n3-2}
{\rm (1)} For any element $a\in\mathcal{D}(V_1)$, we have $\pi_{S'}(l_0)\neq \emptyset$, $\pi_{S'}(a)\neq \emptyset$, and ${\rm diam}_{S'}(\pi_{S'}(l_0), \pi_{S'}(a))\leq 4$.

{\rm (2)} For any element $a\in f(\mathcal{D}(V_2))$, we have $\pi_{S'}(l_2)\neq \emptyset$, $\pi_{S'}(a)\neq \emptyset$, and ${\rm diam}_{S'}(\pi_{S'}(l_2), \pi_{S'}(a))\leq 4$.
\end{lemma}

\begin{proof}
We give a proof for (1) only, since (2) can be proved similarly.
Suppose that $\pi_{S'}(l_0)=\emptyset$ (resp. $\pi_{S'}(a)=\emptyset$). 
This means that for each component $\gamma$ of $l_{0}\cap S'$ (resp. $a\cap S'$), each component of $S'\setminus\gamma$ is an annulus. 
This shows that $S'$ is a sphere with three boundary components, a contradiction. 
If $a=l_0$ or $a\cap l_0=\emptyset$, then we have ${\rm diam}_{S'}(\pi_{S'}(l_0), \pi_{S'}(a))\leq 2$ by Lemma~\ref{subsurface distance}.
Hence, we suppose that $a\ne l_0$ and $a\cap l_0\ne\emptyset$ in the following.

Let $D_a$ be a disk in $V_1$ bounded by $a$, and recall $l_0$ bounds the disk $D_1$ in $V_1$.
We may assume that $|D_a\cap D_1|$ is minimal.
Let $\Delta$ be the closure of a component of $D_a\setminus D_1$ that is outermost in $D_a$.
Let $D_1^{(1)}$ and $D_1^{(2)}$ be the components of $D_1\setminus \Delta$.
By the minimality of $|D_a\cap D_1|$, the disks $D_1^{(1)}\cup \Delta$ and $D_1^{(2)}\cup \Delta$ are essential in $V_1$. 

\begin{claim}
$D_1^{(1)}\cup \Delta$ or $D_1^{(2)}\cup \Delta$, say $D_1^{(1)}\cup \Delta$, is not isotopic to $D_1$ in $V_1$.
\end{claim}

\begin{proof}
Let $m_1$ and $m_2$ be the two simple closed curves obtained from $l_0(=\partial D_1)$ by a band move along $\Delta\cap \partial V_1$.
Suppose both $D_1^{(1)}\cup \Delta$ and $D_1^{(2)}\cup \Delta$ are isotopic to $D_1$ in $V_1$.
This implies that $m_1$ and $m_2$ are parallel in $\partial V_1$, and hence, they co-bound  an annulus, say $A$, in $S$.
Further, by slight isotopy, we may suppose that $l_0\cap (m_1\cup m_2)=\emptyset$.
Note that $l_0$ is retrieved from $m_1\cup m_2$ by a band move along an arc $\alpha$ such that $|\alpha\cap (\Delta\cap \partial V_1)|=1$.
Since $l_0$ is essential, $({\rm int} \alpha)\cap A=\emptyset$.
This shows that $l_0$ cuts off a punctured torus from $\partial V_1$, which contradicts the assumption that $l_0$ is non-separating on $\partial V_1$.
\end{proof}

Hence, by Claim~\ref{claim-n3-1} (1), $l_1$ intersects $D_1^{(1)}\cup \Delta$.
Since $l_1\cap D_1=\emptyset$, $l_1$ intersects $\partial\Delta\setminus D_1$.
Since $l_0\cup l_1$ is separating on $S$, there is a subarc $\gamma$ of $\partial\Delta\setminus D_1$ such that $\partial \gamma \subset l_1$.
Let $\gamma'$ be the closure of a component of $\gamma\setminus N(l_1\cup l_2)$.
Then $\gamma'$ is an element of $\pi_A(a)\,(\subset \mathcal{AC}^0(S'))$.
Hence, we have $${\rm diam}_{\mathcal{AC}(S')}(\gamma', \pi_{A}(a))\leq 1.$$
On the other hand, since $\gamma'$ is disjoint from $l_0$, we have $${\rm diam}_{\mathcal{AC}(S')}(\pi_{A}(l_0), \gamma')\leq 1.$$
By the triangle inequality, we have
\begin{eqnarray*}
\begin{array}{rcl}
{\rm diam}_{\mathcal{AC}(S')}(\pi_{A}(l_0), \pi_{A}(a))&\leq&
{\rm diam}_{\mathcal{AC}(S')}(\pi_{A}(l_0), \gamma')+{\rm diam}_{\mathcal{AC}(S')}(\gamma', \pi_{A}(a))\\
&\leq &1+1=2.
\end{array}
\end{eqnarray*}
By Lemma~\ref{lem-ac-to-c}, we have ${\rm diam}_{S'}(\pi_{S'}(l_0), \pi_{S'}(a))\leq 4$.
This completes the proof of Lemma~\ref{claim-n3-2} (1).
\end{proof}

\begin{lemma}\label{lem-n3}
$V_1\cup_f V_2$ is a strongly keen Heegaard splitting whose Hempel distance is $3$.
\end{lemma}

\begin{proof}
Since $l_0\in \mathcal{D}(V_1)$ and $l_3\in f(\mathcal{D}(V_2))$, we have $d_S(\mathcal{D}(V_1), f(\mathcal{D}(V_2)))\leq 3$.
Let $[m_0, \dots, m_p]$ be a geodesic in $\mathcal{C}(S)$ such that $m_0\in \mathcal{D}(V_1)$, $m_p\in f(\mathcal{D}(V_2))$ and $p\leq 3$.

\begin{claim}\label{claim-5}
$m_i=l_1$ or $m_i=l_2$ for some $i\in\{0,\dots,p\}$.
\end{claim}

\begin{proof}
Assume on the contrary that $m_i\ne l_1$ and $m_i\ne l_2$ for every $i\in\{0,\dots,p\}$.
Namely, every $m_i$ cuts $S'$.
By Lemma~\ref{subsurface distance}, we have 
\begin{equation}\label{eqn-5-1}
{\rm diam}_{S'}(\pi_{S'}(m_0), \pi_{S'}(m_p))\leq 2p\leq 6.
\end{equation}
By the triangle inequality, we have
\begin{eqnarray}\label{eqn-5-2}
\begin{array}{rcl}
{\rm diam}_{S'}(\pi_{S'}(l_0), \pi_{S'}(l_3))&\leq &{\rm diam}_{S'}(\pi_{S'}(l_0), \pi_{S'}(m_0))\\
&&+{\rm diam}_{S'}(\pi_{S'}(m_0), \pi_{S'}(m_p))\\
&&+{\rm diam}_{S'}(\pi_{S'}(m_p), \pi_{S'}(l_3)).
\end{array}
\end{eqnarray}
By the inequalities (\ref{eqn-5-1}), (\ref{eqn-5-2}) together with Lemma~\ref{claim-n3-2}, we obtain
$$
{\rm diam}_{S'}(\pi_{S'}(l_0), \pi_{S'}(l_3))\leq 4+6+4=14,
$$
which contradicts the inequality ${\rm diam}_{S'}(\pi_{S'}(l_0), \pi_{S'}(l_3))>14$ (see Claim~\ref{claim-s'}).
\end{proof}

Assume that $m_i=l_1$ for some $i\in\{0,\dots,p\}$.
(The case where $m_i=l_2$ for some $i\in\{0,\dots,p\}$ can be treated similarly.)
Since $m_0\in\mathcal{D}(V_1)$, we have
$$
d_S(\mathcal{D}(V_1), m_i)\leq d_S(m_0,m_i)=i.
$$
On the other hand, by Claim~\ref{claim-5} and Claim~\ref{claim-n3-1} (1), we have
$$
d_S(\mathcal{D}(V_1), m_i)=d_S(\mathcal{D}(V_1), l_1)=d_S(l_0,l_1)=1.
$$
Hence, we have $i\geq 1$. Similarly, we have $p-i\geq 2$.
These inequalities imply $p=i+(p-i)\geq 1+2=3$.
Hence, $p=3$, and this implies that the Hempel distance of $V_1\cup_f V_2$ is $3$.
Moreover, we have $i=1$, that is, $m_1=l_1$, since, if $i>1$, then $p-i<3-1=2$, a contradiction.

To prove $m_2=l_2$, assume on the contrary that $m_2\ne l_2$. 
Then $m_2$, as well as $m_1(=l_1)$ and $m_3$, cuts $F_2$.
By Lemma~\ref{subsurface distance} and Claim~\ref{claim-n3-1} (3), 
\begin{eqnarray}\label{eqn-5-3}
\begin{array}{rcl}
d_{F_2}(l_1,l_3)&=&d_{F_2}(m_1,l_3)\\
&\leq&{\rm diam}_{F_2}(m_1,\pi_{F_2}(m_3))+{\rm diam}_{F_2}(\pi_{F_2}(m_3), l_3)\\
&\leq&4+4=8,
\end{array}
\end{eqnarray}
which contradicts the inequality $d_{F_2}(l_1,l_3)>8$. 
Hence, $m_2=l_2$.

By Claim~\ref{claim-n3-1} (1), we have $m_0=l_0$ and $m_3=l_3$.
Hence, $[l_0,l_1,l_2,l_3]$ is the unique geodesic realizing the Hempel distance. 
\end{proof}

\section*{Acknowledgement}
We would like to thank Dr Jesse Johnson for many helpful discussions, particularly for teaching us an idea of constructing a unique geodesic path in the curve complex.


\end{document}